\documentclass[12pt, reqno]{amsart}

\usepackage{hyperref}
\usepackage{amsrefs}
\usepackage{setspace}

\numberwithin{equation}{section}
\theoremstyle{plain}
\newtheorem{thm}{Theorem}[section]
\newtheorem{prp}[thm]{Proposition}
\newtheorem{cor}[thm]{Corollary}
\newtheorem{lem}[thm]{Lemma}

\newcommand{\e}{\varepsilon}

\newcommand{\C}{\mathbb{C}}
\newcommand{\R}{\mathbb{R}}
\newcommand{\N}{\mathbb{N}}

\newcommand{\U}{\mathcal{U}}

\newcommand{\dist}{\mathrm{dist}}
\newcommand{\edge}{\mathrm{edge}}

\def\be{\begin{equation}}
\def\ee{\end{equation}}

\begin{document}

\title{Doubling coverings of algebraic hypersurfaces}

\author{Omer Friedland}
\address{Institut de Math\'ematiques de Jussieu, Universit\'e Pierre et Marie Curie (Paris 6), 4 Place Jussieu, 75005 Paris, France.}
\email{omer.friedland@imj-prg.fr}

\author{Yosef Yomdin}
\address{Department of Mathematics, The Weizmann Institute of Science, Rehovot 76100, Israel.}
\email{yosef.yomdin@weizmann.ac.il}
\thanks{Research of the second author was supported by the Israel Science Foundation, grant No. ISF 13/709}

\begin{abstract}
A doubling covering $\U$ of a complex $n$-dimensional manifold $Y$ consists of analytic functions $\psi_j:B_1\to Y$, each function being analytically extendable, as a mapping to $Y$, to a four times larger concentric ball $B_4$.

Main result of this paper is an upper bound on the minimal number $\kappa({\U})$ of charts in doubling coverings of a manifold $Y$, being a compact part of a non-singular level hypersurface $Y=\{P=c\}$, where $P$ is a polynomial on $\C^n$ with non-degenerated critical points. We show that $\kappa({\U})$ is of order $\log({1}/{\rho})$, where $\rho$ is the distance from $Y$ to the singular set of $P$.

Our main motivation is that doubling coverings form a special class of ``smooth parameterizations'', which are used in bounding entropy type invariants in smooth dynamics on one side, and in bounding density of rational points in diophantine geometry on the other. Complexity of smooth parameterizations is a key issue in some important open problems in both areas.

We also present connections between doubling coverings and doubling inequalities for analytic functions $f$ on $Y$, which compare the maxima of $|f|$ on couples of compact domains $\Omega\subset G$ in $Y$. We shortly indicate connections with Kobayashi metric and with Harnack inequality.
\end{abstract}

\maketitle

\section{Introduction}

Let $Y$ be a complex $n$-dimensional manifold, and let $G\subset Y$ be a compact domain in $Y$. Let $B_1$ be the unit ball in $\C^n$. A doubling covering $\U$ of $G$ in $Y$ is a finite collection of analytic univalent functions $\psi_j:B_1\to Y$ satisfying the following conditions:

\smallskip

1) The images (aka charts) $U_j=\psi_j(B_1)$ cover $G$.

\smallskip

2) Each $\psi_j$ is extendible to a mapping $\tilde\psi_j:B_4\to Y$, which is univalent in a neighborhood of $B_4$, where $B_4\subset\C^n$ is the four times larger concentric ball of $B_1$. If $B_4$ is replaced by $B_\gamma, \gamma > 1$, then the covering is called $\gamma$-doubling.

\smallskip

Doubling coverings provide, essentially, a conformally invariant version of the Whitney's ball coverings of a domain $W\subset {\mathbb R}^n$, introduced in \cite{Whi}. These coverings consist of balls $B_j$ such that larger concentric balls $\gamma B_j$ are still in $W$ (compare Section \ref{Sec:Zig.Cald} below). In our definition we replace $W$ by a complex manifold $Y$, while the balls $B_j$ are replaced by the charts $U_j$. See a survey in \cite[Chapter 6]{Bru.Bru} for extensions and developments of Whitney coverings in other directions.

\smallskip

Introduction and study of doubling coverings in this paper is motivated mostly by the fact that they form a special class of ``smooth parameterizations'', which are used in bounding entropy type invariants in smooth dynamics on one side, and in bounding density of rational points in diophantine geometry on the other.
In fact, a doubling covering is a complex counterpart of an ``analytic parameterization'', as introduced in \cite{Yom3} and further developed in \cite{Yom4, Yom5}. Other types of ``smooth parameterizations'' are $C^k$-ones, introduced in \cite{Yom1, Yom2} and studied in \cite{Bur, Gro, PW} and in other publications, where the size of the derivatives up to order $k$ is controlled. In ``mild parameterizations'', introduced in \cite{Pil1} and further studied in \cite{Pil2, Tho1}, and others, the growth rate of the derivatives up to infinity is controlled.
There are prominent open problems in dynamics and in diophantine geometry (see, e.g. \cite{BLY, PW, Yom3, Yom5} and references therein), where constructing analytic and mild parameterizations, and bounding their complexity, are expected to be important. Very recently an important progress in these problems was achieved in \cite{Bin.Nov1,Bin.Nov2,Clu.Pil.Wil}, in particular, via introducing a new type of ``ramified analytic parameterization''. We expect that these results will strongly contribute to a better understanding of various types of smooth parameterizations, and of their mutual relations. 

\smallskip

Let $Y$ be a non-singular level hypersurface
$$
Y=Y_c=\{P=c\}
$$
where $P$ is a polynomial of degree $d\ge2$ on $\C^n$, with only isolated and non-degenerate critical points. Let $G_c=Y_c\cap Q$, where $Q=Q^n$ is a unit cube in $\C^n$, and let $\delta>0$ be the distance of $Y_c$ to the critical points of $P$. We are interested in doubling coverings $\U$ of $G_c$ in $Y_c$, as $c$ approaches a certain critical value of $P$. In this case $\delta \to 0$, and the geometric complexity of $Y_c$ (in particular, its curvature) near the critical points of $P$ ``blows up''. One can expect that the minimal number $\kappa(\U)$ of charts in doubling coverings $\U$ of $G_c$ in $Y_c$ also tends to infinity. However, this problem turns out to be rather delicate: it was shown in \cite{Gro} that for each fixed smoothness $k$ the minimal number of charts in $C^k$-parameterizations of $Y_c$ remains uniformly bounded, in terms of $n$ and $d$ only.

\smallskip

The main result of this paper is an upper bound on the complexity $\kappa(\U)$ of a doubling covering $\U$ of $G_c$ in $Y_c$ of the form 
$$
\kappa(\U) \leq C(P)\log(\frac{1}{\delta}).
$$
In some special cases we provide also the lower bound for $\kappa(\U)$, of the same form. So for doubling analytic coverings, in a strict contrast with $C^k$-parameterizations, their complexity, at least in some special cases, grows as a logarithm of the distance to complex singularities. We conjecture that this result remains true also for polynomials $P$ with possibly degenerate (and non-isolated) singularities.

\smallskip
 
As the second main topic of this paper, we present various types of doubling inequalities, and demonstrate a very general explicit connection between them and chains of charts in doubling coverings on $Y$, closely resembling a well-known construction, applied in Harnack-type inequalities.

\smallskip

Let $\Omega\subset G$ be compact domains in $Y$. Let $f$ be an analytic function in a neighborhood of $G$ in $Y$, the doubling constant of $f$ with respect to $\Omega$ and $G$ is the ratio
$$
DC_f(G, \Omega)={\max_{G}|f(z)|}/{\max_{\Omega}|f(z)|}.
$$

Doubling inequalities provide an upper bound on $DC_f(G, \Omega)$ for various classes of analytic functions $f$ on $Y$.

\smallskip

In recent years they have been intensively studied for algebraic functions, in connection with various problems in harmonic analysis and potential theory, differential equations, diophantine geometry, probability, complexity, etc. (see e.g. \cite{Bru,FN,RY} and references therein).
However, in these results the variety $Y$ on which the doubling inequalities are considered, is usually fixed, while the degree of the restricted polynomials grows.
An important question of the dependence of the doubling constant on $Y$ (in particular, in families like $Y_c$ above), remaind largely open. Using our main result on the complexity of the doubling coverings of $Y_c$, we show that for polynomials $S$ of degree $d_1$ restricted to $Y_c$ we have
$$
DC_S(G, \Omega) \leq (\frac{C_1}{\delta})^{C_2},
$$
with $C_1,C_2$ depending on $P$, and on the degree $d_1$ of the restricted polynomial $S$.

\smallskip

We believe that the results of the present paper provide a step towards a better understanding of the behavior of the doubling constant in families of varieties $Y$. In particular, we expect that the assumption of non-degenerate singularities of $P$ can be dropped, while preserving the polynomial dependence of the doubling constant on $\delta$.

\smallskip

The paper is organized as follows: in Section \ref{Sec:More.Doub.Cov} we introduce some notations and definitions with respect to doubling coverings, and discuss their connection to Kobayashi distance. The complexity of doubling coverings of $Y$ depends only on a complex analytic structure of $Y$, and so one can hope to define, in its terms, certain invariants of $Y$. We make an initial step in this direction, showing that the length of chains (or the total complexity $\kappa({\U}))$ in doubling coverings $\U$, bounds the Kobayashi distance on $Y$. We also give a special example of a doubling covering {\it with balls} of a punctured cube, which is obtained via a construction in the spirit of Whitney's covering lemma (\cite{Whi}, see also \cite{St}), which provides a bound, depending only on the number of the removed points, but not on their position. We also discuss briefly the behavior of doubling chains with respect to Harnack-type inequalities for harmonic functions.

\smallskip

Section \ref{Sec:Cov.Alg.Hyp} is central in our presentation, we prove here our main result, providing ``controlled'' doubling coverings for algebraic hypersurfaces. We provide an explicit construction of a doubling covering $\U$ of $G$ in $Y$, with the number of charts $\kappa({\U})$ of order $C(n,d)\log({c(n,d)}/{\rho})$. It is based on the special example of a doubling covering {\it with balls} of a punctured cube, and a quantitative implicit function theorem that we also provide.

\smallskip

Section \ref{Sec:Doub.Ine} is of technical nature, preparing some basic results on doubling inequalities for $p$-valent functions on balls.

\smallskip

In Section \ref{Sec:Gen.Doub.Ine} we give a very general form of a doubling inequality for analytic functions $f$ on a manifold $Y$ in terms of doubling coverings $\U$ of $Y$. We do so via continuation along chains of charts in $\U$. As a consequence of our upper bound on $\kappa({\U})$, and this general doubling inequality, we obtain a doubling inequality on $Y$ for the restrictions to $Y$ of polynomials of certain degree.

\smallskip

Finally, Section \ref{Sec:examples} provides some examples. In particular, inverting the inequality of Section \ref{Sec:Gen.Doub.Ine}, and presenting specific polynomials $S$ with a large doubling constant on $Y$, we obtain, in certain specific cases, a lower bound, of the same order $\log({1}/{\delta})$, on the number $\kappa({\U})$ of charts in doubling coverings $\U$ of $Y$.

\section{Doubling coverings} \label{Sec:More.Doub.Cov}

In this section we introduce some notations and definitions with respect to doubling coverings, and discuss their connection to Kobayashi distance and Harnack-type inequalities. We also construct a special example of a doubling covering of a punctured cube in $\C^n$.

\smallskip

Let $Y$ be a complex $n$-dimensional manifold, let $G\subset Y$ be a compact domain in $Y$, and let $\U$ be a doubling covering of $G$ in $Y$.

\smallskip

We call two charts $U_i$ and $U_j$ in $\U$ neighboring, if $U_i\cap U_j\not=\emptyset$. For two neighboring charts $U_i$ and $U_j$ we define the intersection radius $\rho(U_i, U_j)$ as the maximal radius $\rho >0$ such that both $\psi^{-1}_i(U_i\cap U_j)\subset B_1\subset\C^n$, and $\psi^{-1}_j(U_i\cap U_j)\subset B_1\subset\C^n$ contain subballs of radius $\rho$ (not necessarily concentric with $B_1$). In a similar way for $\Omega\subset G$ a subdomain in $G$, and a chart $U_j\in\U$ we define $\rho(U_j, \Omega)$ as the maximal radius $\rho>0$ such that $\psi^{-1}_j(U_j\cap\Omega)\subset B_1$ contains a subball of radius $\rho$. We put $\rho(\U, \Omega)=\min_{j}\rho(U_j, \Omega)$.

\smallskip

A chain $Ch$ in a covering $\U$ is a set $\{j_1, j_2, \dots, j_n\}$ of pairwise different indices, such that $U_{j_p}$, $U_{j_{p+1}}$ are neighboring for each $p=1, \dots, n-1$. The length $n$ of the chain $Ch$ is denoted by $\ell(Ch)$. The collection $CH(z, \Omega, {\U})$ consists of all the chains $Ch=\{j_1, j_2, \dots, j_n\}$ in $\U$ such that $\rho(U_{j_1}, \Omega)>0$, while $z\in U_{j_n}$.

\subsection{Doubling coverings and Kobayashi metric} \label{Sec:Kobayashi}

Let $Y$ be a complex $n$-dimensional manifold, and let $p, q\in Y$. The Kobayashi distance (or, more accurately, pseudo-distance) $d(p, q)$ is defined as follows \cite{Kob}: choose points $p=p_0, p_1, \dots, p_{k-1}, p_k=q\in Y$, points $a_1, \dots, a_k$, $b_1, \dots, b_k$ in the unit disk $D_1\subset \C$, and holomorphic mappings $f_1, \dots, f_k$ from $D_1$ to $Y$, such that $f_i(a_i)=p_{i-1}$, $ f_i(b_i)=p_{i}$, $i=1, \dots, k$. Form a sum $\sum_{i=1}^k\rho(a_i, b_i)$, where $\rho$ is a Poincar\'e metric on $D_1$, and put $d(p, q)$ to be the infimum of these sums for all possible choices.

\begin{prp} \label{Kobayashi}
Let $G\subset Y$ be a connected compact domain, and let $p, q\in G$. Let $\U$ be a doubling covering of $G$ in $Y$, and let $Ch$ be a chain in $\U$ joining $p$ and $q$. Then, the Kobayashi distance $d(p, q)$ satisfies
$$
d(p, q)\le 3\ell(Ch)\le 3\kappa({\U}).
$$
\end{prp}

\begin{proof}
Let $U_1, \dots, U_l$ be the charts in $Ch$. Denote $p_0=p$, $p_l=q$, and for $i=1, \dots, l-1$ pick $p_i$ to be a point in $U_i\cap U_{i+1}$. Next, put $\tilde a_i=\psi_i^{-1}(p_{i-1})$, and $\tilde b_i=\psi_i^{-1}(p_{i})$ in $B_1$. Now, define an affine map $T_i:D_1\to B_4$, requiring the image $T_i(D_1)$ be the intersection disk $\tilde D$ of $B_4$ and of the complex line, passing through the points $\tilde a_i, \tilde b_i\in B_1$. Clearly, the radius of $\tilde D$ is at least $3$, while the points $\tilde a_i, \tilde b_i\in B_1$ belong to a concentric subdisk of $\tilde D$ of radius at most $1$.

\smallskip

Finally, we put $a_i=T_i^{-1}(\tilde a_i)$, and $b_i=T_i^{-1}(\tilde b_i)$, and take $f_i=\psi_i\circ T_i$. It remains to notice that for each $i$ our points $a_i, b_i$ belong to the concentric disk $D_{1/3}$ of radius $\frac{1}{3}$ in $D_1$, and hence $\rho(a_i, b_i)\le 3/2$. Indeed, the Poincar\'e metric on $D_1$ is given by $ds=\frac{2|dz|}{1-|z|^2}$. So inside $D_{1/3}$ we have $ds \le 9/4 |dz|$, and therefore the Poincar\'e distance $\rho(a_i, b_i)$ does not exceed $3/2$.
\end{proof}

\subsection{Doubling coverings and Harnack-type inequalities} \label{Sec:Harnack}

``Extension along chains'' of doubling charts, which we use in Section \ref{Sec:Doub.Ine} in order to obtain doubling inequalities, is one of the classical and widely used tools in study of Harnack-type inequalities for harmonic functions and, more generally, for solutions of certain classes of PDE's (see, e.g. \cite{Aik} and references therein).

\smallskip

There is, however, an essential difference: usually, only coverings with balls are used. The reason is that a general complex analytic change of variables preserves harmonic functions only in complex dimension one. In the case of two or more variables already linear changes of variables, if not dilations, destroy the condition $\Delta f =0$.

\smallskip

In our context, in Section \ref{Sec:Zig.Cald} below a doubling covering {\it with balls} is constructed for the punctured cube in $\R^n$. We plan to extend this construction to the complements of algebraic varieties of higher dimensions, and apply it to Harnack-type inequalities for harmonic functions in \cite{Prep}. However, the doubling charts on the level hypersurfaces $Y$ constructed in Section \ref{Sec:Cov.Alg.Hyp} are nonlinear. So these charts cannot be applied to Harnack inequalities directly.

\subsection{$\gamma$-doubling ball covering of a punctured cube} \label{Sec:Zig.Cald}

In this section we construct a $\gamma$-doubling ball covering of a punctured cube, where $\gamma>1$ is the doubling factor. Our construction is inspired by the classical Whitney's covering lemma (\cite{Whi}). A similar construction appears also in Calder\'on-Zygmund decomposition (e.g. see \cite{St}). In fact, our construction works in the real space $\R^n$, and provides a covering with Euclidean balls. Notice that a connection between the geometry of a closed set and of its tabular neighborhoods and counting Whitney cubes in the complement is well known (see \cite{Har.Hur.Vah,Mar.Vuo,Kae.Leh.Vuo} and references therein). We provide here an explicit (non-asymptotic) counting of the Whitney cubes, covering a $\delta$-punctured unit cube, with the bound depending on the number of the deleted points, but not on their mutual position. We also describe explicitly the intersections of the corresponding covering balls.

\smallskip

Let $W\subset \R^n$ be an open domain, and let $G \subset W$ be a compact set. A $\gamma$-doubling ball covering $\U$ of $G$ in $W$ is a collection of balls $B_j\subset W$, which covers $G$ such that the concentric balls $\gamma B_j$ are contained in $W$.

\smallskip

In case when $\R^{2n}$ is the underlying real space of the complex space $\C^n$, any $\gamma$-doubling ball covering $\U$ is a complex $\gamma$-doubling covering, with the mappings $\psi_j$ being the linear scaling mappings of $B_1$ to $B_\gamma$.

\smallskip

Let $Q=[-1, 1]^n \subset \R^n$ be the $n$-dimensional unit cube, and let $z_1, \dots, z_d\in \R^n$. Denote by $U_\delta$ a $\delta$-neighborhood of $\{z_1, \dots, z_d\}$, and consider the domain $Q_\delta=Q\setminus U_\delta$, that is, we removed from $Q$ balls of radius $\delta>0$ around each point $z_1, \dots, z_d$.

\begin{thm} \label{thm:beta-doubling-covering}
Let $\gamma >1$. There is a $\gamma$-doubling ball covering $\U$ of $Q_\delta$ in $\R^n\setminus \{z_1, \dots, z_d\}$ with at most
$$
d(3\sqrt n\gamma)^n \log(3n\gamma/\delta)
$$
balls.

\smallskip

Moreover, for any $v, w$ belonging to the same connected component of $Q_\delta$, there exists a chain $Ch$ in $\U$, joining $v$ and $w$, such that for any two consequitive balls $B_{j_1}$ and $B_{j_2}$ in $Ch$ the ratio of the radii of these balls is either $\frac12$, $1$ or $2$, and the intersection $B_{j_1}\cap B_{j_2}$ contains a ball of the radius at least $1/3$ of the smaller of the radii.
\end{thm}

\begin{proof}
We construct the balls $B$ in the required $\gamma$-doubling covering $\U$ as the circumscribed balls of certain sub-cubes in binary subdivisions of $Q$ such that for any $B \in \U$
\begin{align} \label{doubling-cond}
\{z_1, \dots, z_d\} \cap \gamma B = \emptyset.
\end{align}

For $s=1, 2, \dots$ we call the closed sub-cubes $Q_s$, obtained by a subdivision of $Q$ into $2^{ns}$ equal parts, the level-$s$ sub-cubes, i.e. level-$s$ sub-cubes are derived by subdividing level-$(s-1)$ sub-cubes into $2^n$ parts. Since some of the points $\{z_1, \dots, z_d\}$ may be out of $Q$, we extend this subdivision to the entire space $\R^n$. We say that two sub-cubes $Q_l^q, Q_r^p$ are neighbors if $Q_l^q\cap Q_r^p \not=\emptyset$, and the $k$-neighborhood of a given level-$s$ sub-cube consists of all its neighbor level-$s$ sub-cubes up to ``distance'' $k$. Naturally, a $k$-neighborhood contains at most $(2k+1)^n$ level-$s$ sub-cubes (for any $s$). The length of an edge of a level-$s$ sub-cube is $\edge_s=2/2^s$, and the radius of the corresponding ball is $r_s=\sqrt{n\left(2/{2^s}\right)^2}/2=\sqrt n/2^s$. With these notation in hand, and in view of condition \eqref{doubling-cond}, we wrap each point $z_j$ with its $k$-neighborhood so that the following inequality holds
\be \label{inq-cond}
\gamma r_s \le k \cdot \edge_s + \edge_s/2
\ee
which is satisfied for
\be \label{eq:k}
k=\left\lceil\frac{\sqrt n\gamma-1}{2}\right\rceil.
\ee

In other words, assume that $z_j\in Q_s$ then for any circumscribed ball $B$ of a level-$s$ sub-cube outside the $k$-neighborhood of $Q_s$ we have $\gamma B \cap Q_s=\emptyset$.

\smallskip

Now, assume that in step $s$ the collections $S_1, \dots, S_s, \Sigma_s$ of sub-cubes inside $Q$ have been constructed, with the following properties:

\smallskip

1) $S_l$ consists of certain level-$l$ sub-cubes $Q_l$ inside $Q$, their number is at most $(2k+1)^n 2^n d$. Denote by $B_l$ the circumscribed ball of $Q_l\in S_l$, then $B_l$ satisfies condition \eqref{doubling-cond}, that is, $\{z_1, \dots, z_d\} \cap \gamma B_l=\emptyset$.

\smallskip

2) The sub-cubes of $S_l$ for $1\le l \le s-1$ may have neighbors only from $S_{l-1}, S_{l},S_{l+1}$, while sub-cube of $S_s$ may have neighbors from $S_{s-1},S_s, \Sigma_s$. Moreover, the ``$s$-distance'' between any sub-cube in $\Sigma_s$ to sub-cubes in $S_{s-1}$ is at least $k$.

\smallskip

3) $\Sigma_s$ consists of exactly those level-$s$ sub-cubes $Q_s$ inside $Q$, which either contain some points of $\{z_1, \dots, z_d\}$, or which are in (level-$s$) $k$-neighborhoods of certain $\tilde Q_s$ containing points $\{z_1, \dots, z_d\}$. The number of sub-cubes in $\Sigma_s$ is at most $(2k+1)^nd$. The sub-cubes in $\Sigma_s$ may have non-empty intersection only with level-$s$ sub-cubes of $S_s,\Sigma_s$.

\smallskip

4) The collections $S_1, \dots, S_s, \Sigma_s$ are disjoint and all their sub-cubes form a covering of $Q$.

\smallskip

{\it The induction step}. We proceed as follows, we subdivide each $Q_s \in \Sigma_s$ into $2^n$ equal sub-cubes $Q_{s+1}$. Altogether we get at most $(2k+1)^n 2^n d$ sub-squares.

\smallskip

Let $\Sigma_{s+1}$ be the union of those $Q_{s+1}$, which either contain some points of $\{z_1, \dots, z_d\}$, or which are in (level-$s$) $k$-neighborhoods of certain $\tilde Q_s$ containing points $\{z_1, \dots, z_d\}$. The number of sub-cubes in $\Sigma_{s+1}$ is at most $(2k+1)^n\cdot d$.

\smallskip

The sub-cubes in $\Sigma_{s+1}$ have non-empty intersection only with level-$(s+1)$ sub-cubes. Indeed, by property $3$ $\Sigma_s$ consists of all the level-$s$ sub-cubes $Q_s$ inside $Q$, which are in $s$-distance at most $k$ from the sub-cubes containing points $\{z_1, \dots, z_d\}$. After subdividing, the new level-$(s+1)$ $k$-neighborhood of $\{z_1, \dots, z_d\}$ is of $(s+1)$-distance at least $2\cdot k-k=k$ from any level-$s$ sub-cubes in $S_s$, and the in-between sub-cubes are of level-$(s+1)$ (these sub-cubes actually belong to $S_{s+1}$, as we shall see below). This proves, for $\Sigma_{s+1}$, property $3$ and the last part of property $2$.

\smallskip

Let $S_{s+1}$ be the union of the remaining $Q_{s+1}$, their number is at most $(2k+1)^n 2^n d$. Clearly, $S_{s+1},\Sigma_{s+1}$ are disjoint. Each $Q_{s+1}$ may have non-empty intersection with level-$s$ sub-cubes of $S_s$, and with level-$(s+1)$ sub-cubes of $S_{s+1},\Sigma_{s+1}$ (as we subdivide $\Sigma_s$ and it has neighbors from $S_s$). It also means that now, after subdivision, sub-cubes of $S_s$ may have non-empty intersection with sub-cubes of $S_{s+1}$. However, sub-cubes of $S_{s+1}$ cannot intersect sub-cubes of $S_{s-1}$. Indeed, they appear in subdivision of sub-cubes in $\Sigma_s$, which are at $s$-distance from $S_{s-1}$ at least $k$, by the last part of property $2$.

\smallskip

For each $Q_{s+1} \in S_{s+1}$ we build the circumscribed concentric ball $B_{s+1}$. By inequality \eqref{inq-cond} and the choice of $k$ (i.e. the construction of $\Sigma_{s+1}$) the concentric ball $\gamma B_{s+1}$ does not contain the points $z_1, \dots, z_d$. This completes the proof of properties $1$ and $2$.

\smallskip

We've subdivided only sub-cubes in $\Sigma_s$. So, the collections $S_1,\dots,S_s,S_{s+1}$, $\Sigma_{s+1}$ are disjoint, and their sub-cubes form a covering of $Q$. Note that it could be that $S_1,\dots,S_t$ are empty, for $t$ which satisfies $2^t \le 2k+1$. This completes the proof of property $4$ and the induction step.

\smallskip

Now, we complete the proof of Theorem \ref{thm:beta-doubling-covering}. By property $1$, each $S_l$ contains at most $(2k+1)^n2^nd$ sub-cubes of level $l$. Hence, the total number of sub-cubes in
$$
\mathcal{S}_s = S_1 \cup \dots \cup S_s
$$
is at most $s(2k+1)^n2^nd \le sd(3\sqrt n\gamma)^n$ (by the choice of $k$ in \eqref{eq:k}).

\smallskip

On the other hand, if $\Sigma_s\subset U_\delta$ then the process stops, and thus the collection of the circumscribed balls of the sub-cubes in $\mathcal{S}_s$ provides a $\gamma$-doubling covering $\U$ of $Q_\delta$. So, if the maximal possible distance $\frac12\sqrt{n\left((2k+1)\frac{2}{2^s}\right)^2}$ of points in $\Sigma_s$ from $\{z_1, \dots, z_d\}$ is equal to $\delta$, or $s = \log(3n\gamma/\delta)$ then $\Sigma_s$ is contained in $U_\delta$, and for this value of $s$ the process stops. Therefore, the total number of sub-cubes in $\mathcal{S}_s$ is at most
$$
d(3\sqrt n\gamma)^n \log(3n\gamma/\delta) .
$$

It remains to find, for any $v, w$ in the same connected component of $Q_\delta$, a chain $Ch$ in $\U$ joining $v$ and $w$, such that for any two neighbor balls $B_{j_1}$ and $B_{j_2}$ in $Ch$ the ratio of the radii of these balls is either $1$ or $2$, and the intersection $B_{j_1}\cap B_{j_2}$ contains a ball of radius at least $1/3$ of the smaller of the radii.

\smallskip

We construct a chain $Ch$ in $\U$ joining $v, w$ along a certain continuous path $\omega$ in $Q_\delta$. We can assume that $\omega$ intersects only with the interiors of the subdivision sub-cubes, and with their faces of dimension $n-1$, but does not touch faces of smaller dimensions. Since the sub-cubes in $\mathcal{S}_s$ form a covering of $Q_\delta$, following $\omega$, taking the subsequent sub-cubes, crossed by $\omega$, and omitting possible repetitions, we find a chain $\hat Ch$ of sub-cubes in $\mathcal{S}_s$ joining $v$ and $w$, in which any two neighboring sub-cubes not only have a non-empty intersection, but, in fact, intersect along a part of their common face of dimension $n-1$. Now, we define $Ch$ as the chain of the circumscribed balls for $\hat Ch$.

\smallskip

By property $1$ above the levels of the neighboring sub-cubes in $\hat Ch$ may differ at most by one. Hence, the ratio of the radii of the corresponding circumscribed balls is either $\frac12$, $1$ or $2$. Indeed, the ratio between two neighbor balls of level $l$ and $l+1$ is
$$
r_l/r_{l+1} = (\sqrt n/2^l)/(\sqrt n/2^{l+1})=2.
$$

Finally, an easy geometric calculation shows that in case when the neighboring sub-cubes in $\U$ intersect along a part of their common face of dimension $n-1$, the intersection of the circumscribed balls balls contains a ball of the radius at least $1/3$ of the smaller of the radii. Consider such sub-cubes of levels $s$ and $s+1$, the case of sub-cubes of the same level being completely similar. Now, the largest distance between the centers of two neighbor sub-cubes of level $s$ and $s+1$ with a common face is obtained (after putting in a standard position) for $Q_s$ with center at $A=(\frac12\frac{2}{2^s},\dots,\frac12\frac{2}{2^s})$, and $Q_{s+1}$ with center at $B=(\frac{2}{2^s}+\frac12\frac{2}{2^{s+1}},\dots,\frac12\frac{2}{2^{s+1}})$, where $Q_{s+1}$ is placed in a corner of an $(n-1)$-dimensional face of $Q_s$
\begin{align*}
\|A-B\| &= \sqrt{\left(\frac12\frac{2}{2^s}-\left(\frac{2}{2^s}+\frac12\frac{2}{2^{s+1}}\right)\right)^2+(n-1)\left(\frac12\frac{2}{2^s}-\frac12\frac{2}{2^{s+1}}\right)^2} \\
& = \frac{\sqrt{n+8}}{2^{s+1}} .
\end{align*}

The maximal radius $r$ of a ball, which can be placed inside the intersection of the corresponding circumscribed balls of $Q_s,Q_{s+1}$ is given by
$$
r = \frac12(r_s + r_{s+1}-|A-B|) = \frac{3\sqrt n-\sqrt{n+8}}{2^{s+2}} .
$$

Thus, the ratio of the radii of this ball and of a ball of level-$s+1$ is
$$
r/r_{s+1} = \frac{3\sqrt n-\sqrt{n+8}}{2^{s+2}} \frac{2^{s+1}}{\sqrt n} = \frac32-\sqrt{\frac14+\frac{2}{n}} \ge \frac13.
$$

This completes the proof of Theorem \ref{thm:beta-doubling-covering}.
\end{proof}

Note that the bound of Theorem \ref{thm:beta-doubling-covering} is sharp with respect to the parameters $d$, $\gamma$ and $\delta$, up to coefficients depending only on the dimension $n$. Consider the case of only one point $z_1=0\in \R^n$, and let $\U$ be a $\gamma$-doubling ball covering of $Q_\delta$ in $\R^n \setminus \{0\}$. Each ball $B$ in $\U$, centered at $z_0$ of radius $R$ satisfies $R\gamma < ||z_0||$. So to cover a spherical shell $\gamma R \le ||z|| \le \gamma R+1$ we need at least $C_1(n)\gamma^{n-1}$ balls in $\U$. Now, to ``reach'' the $\delta$-neighborhood of $0$ we need $\log (C_2(n)\gamma/\delta)$ concentric spherical shells as above. Finally, for several points $z_1,\ldots,z_d$, and for $\delta$ small enough, we can apply the above considerations to each point $z_j$ separately. Altogether we obtain a lower bound for $\kappa(\U)$ of the form $d C_1(n)\gamma^{n-1}\log (C_2(n)\gamma/\delta)$.

\smallskip

\noindent{\bf Remark 1.} {\it In the construction of the chain above it was not necessary to require the subsequent sub-cubes to have a common part of an $(n-1)$-face. Instead we could require them only to intersect by more than a vertex. This would just provide an absolute bound smaller than $1/3$ for the radius of the ball in the intersection.}

\smallskip

\noindent{\bf Remark 2.} {\it Theorem 3 of \cite{Vol.Kon} compares the multiplicities of a covering with certain balls, and with the twice larger concentric balls. It would be interesting to see implications of this result for Whitney-type doubling coverings.}

\section{Covering algebraic hypersurfaces} \label{Sec:Cov.Alg.Hyp}

Let $P(z)=\sum_{|\alpha|\le d}a_\alpha z^\alpha$ be a polynomial of degree $d$ on $\C^n$, with the usual multi-index notations: for $z=(z_1, \dots, z_n)\in \C^n$, and for $\alpha=(\alpha_1, \dots, \alpha_n)\in\N^n$, we have $|\alpha|=\sum_{i=1}^n \alpha_i$ and $z^\alpha=z_1^{\alpha_1} \cdots z_n^{\alpha_n}$. We denote the $\ell_1$-norm of $P$ by $\|P\|_1=\sum_{|\alpha|\le d}|a_\alpha|$.

\smallskip

The complex singular set $\Sigma=\Sigma(P)$ of the polynomial $P$ is defined by vanishing of all the partial derivatives $\frac{\partial P}{\partial z_j}$, $j=1, \dots, n$. For a generic $P$ its singular set consists of isolated and non-degenerate critical points: $\Sigma(P)=\{w_1, \dots, w_m\}$. By B\'ezout theorem $m\le (d-1)^n$, and a strict inequality may happen if some of singular points of $P$ are at infinity.

\smallskip

In what follows we always assume that the polynomial $P$ is normalized, i.e. $\|P\|_1=1$, and that {\it all the affine complex singular points $w_1, \dots, w_m$ of $P$ are isolated and non-degenerate}. In particular, this assumption implies, for $z\in Q$ (now $Q$ denotes the \textit{complex} $n$-dimensional unit cube), the following inequality
\be \label{eq:grad.P.dist}
K(P) \dist(z, \Sigma) \le \|\nabla P(z)\| \le n d^4 \dist(z, \Sigma)
\ee
where the gradient $\nabla P(z)=(\frac{\partial P(z)}{\partial z_1}, \dots, \frac{\partial P(z)}{\partial z_n})$, $K(P)>0$ is a constant depending on $P$, and $\|\cdot\|$ is the usual Euclidean norm of the gradients. The upper bound for $\|\nabla P(z)\|$ easily follows from Markov's inequality: the second partial derivatives of $P$ are bounded for $z\in Q$ by $n d^4 \|P\|_1=n d^4$. Integrating along the straight segment from $z$ to the nearest point in $\Sigma$ we obtain $\|\nabla P(z)\|\le n d^4 \dist(z, \Sigma)$. However, the bound from below for $\|\nabla P(z)\|$ of the form \eqref{eq:grad.P.dist} is valid only under our ``general position'' assumption.

\smallskip

Let us stress that the constant $K(P)$ in \eqref{eq:grad.P.dist} depends not only on the degree of the polynomial $P$, but {\it on its specific coefficients}. It can be bounded from below in terms of the minimal eigenvalues of the Hessians of $P$ at the critical points in $\Sigma(P)=\{w_1, \dots, w_m\}$. To simplify the presentation, we just take $K(P)$ as an explicit input parameter. Notice that by (\ref{eq:grad.P.dist}) we always have $K(P)\leq n d^4.$

\smallskip

We consider complex algebraic hypersurfaces $Y$, which are the level sets of $P$
$$
Y=\{P(z)=c\} \subset \C^n
$$
where $c$ is assumed to be a regular value of $P$. Thus, $Y$ is a nonsingular submanifold of dimension $d-1$ in $\C^n$.

\begin{thm} \label{thm:main.hypers.1}
Let $P(z)$ be a normalized polynomial on $\C^n$ of degree $d\ge2$, with isolated and non-degenerate critical points $\Sigma=\Sigma(P)=\{w_1, \dots, w_m\}$, so $P$ satisfies condition \eqref{eq:grad.P.dist} with $K=K(P)$.

\smallskip

Let $Y=\{P(z)=c\}$ be a regular level hypersurface of $P$. Denote $G=Y\cap Q$, and put $\delta=\dist(G, \Sigma(P))>0$. Then, there exists a doubling covering $\U$ of $G$ in $Y$ with $\rho({\U})\ge\frac{1}{10}$ and
$$
\kappa({\U})\le \frac{C_1(n, d)}{K^{2n}}\log(\frac{C_2(n, d)}{K\delta})
$$
where the constants $C_1(n, d), C_2(n, d)$ depend only on $n, d$.
\end{thm}

\begin{proof}
The main steps of the proof are as follows: as usual in differential topology, we produce doubling covering charts of $Y$ (which are special coordinate charts), using implicit function theorem. However, to count these charts, we need a ``quantitative'' version of this theorem, stated below. It produces, at a given point $z\in Y$, a coordinate chart of the size proportional to the norm of the gradient $\nabla P(z)$. So, near the critical points of $P$ we need more charts. By our assumptions the hypersurface $Y$ is at the distance at least $\delta$ from the critical set $\Sigma=\Sigma(P)$. Hence, it is contained in $Q_\delta=Q\setminus\Sigma_\delta$, where $\Sigma_\delta$ is the $\delta$-neighborhood of $\Sigma$. So, in order to control the construction explicitly, we apply Theorem \ref{thm:beta-doubling-covering}, with $\gamma$ of order $1/K$, and obtain a $\gamma$-doubling ball covering $\U$ of $Q_\delta$, with $C(n,d,\gamma)\log({c(n,d,\gamma)}/{\delta})$ charts. Because of the assumption \eqref{eq:grad.P.dist} on $P$, in each ball of $\U$ we get a lower bound on the norm of $\nabla P(z)$.

\smallskip

Now, we present the construction in detail. We use the following notations: $B_R=B^n_R$ is a complex ball of dimension $n$ and of radius $R$. Assuming that a coordinate system $z_1, \dots, z_n$ in $\C^n$ is fixed, we consider $\C^{n-1}\subset \C^n$ corresponding to the first $n-1$ coordinates $z_1, \dots, z_{n-1}$. For $z=(z_1, \dots, z_n)$ in $\C^n$ we denote $\bar z=(z_1, \dots, z_{n-1})$ its projection to $\C^{n-1}$, and we denote $\bar B_R=B^{n-1}_R$ a complex ball of dimension $n-1$ and of radius $R$. Finally, a ``diskoball'' $DB_R$ is a product $\bar B_R\times D_R$ of the ball of radius $R$ with respect to the first $n-1$ coordinates, and a disk of radius $R$ with respect to the last coordinate. We shall use a certain quantitative version of the standard implicit function theorem. Various settings of this result are known (see e.g. \cite{BM,Yom6}), however, they do not cover exactly the specific statements of the result below, so we provide a short proof.

\begin{thm} \label{thm:impl.fn}
Let $f(z_1, \dots, z_n)$ be a complex analytic function on an open (sufficiently large) domain $W\subset\C^n, 0\in W$. Assume that $f (0)=0$, $\frac{\partial f(0)}{\partial z_j}=0$ for $j=1, \dots, n-1$, while $|\frac{\partial f(0)}{\partial z_n}|=\eta >0$. Assume also that the supremum on $W$ of the absolute value of the second partial derivatives of $f$ does not exceed $M$. Put $\theta=\frac{\eta}{50M\sqrt {2n(n-1)}}$. Then, the following properties hold:

\smallskip

1) Inside the diskoball $DB_\theta$ centered at $0\in\C^n$ the set of zeroes of $f$, i.e. $Y=\{z: f(z)=0\}$, is a regular analytic hypersurface, which is the graph of a regular analytic function $z_n=\phi(z_1, \dots, z_{n-1})$, with $\|\nabla\phi\|\le\frac{1}{49}$ on $\bar B_\theta$.

\smallskip

2) Inside the diskoball $DB_\theta$ the hypersurface $Y$ is contained in a tabular neighborhood $W_\nu$ of the coordinate hyperplane $z_n=0$, of the size $\nu=\frac{\theta}{49}$. The projection $\pi: Y\to\C^{n-1}$, $\pi(z)=\bar z$, restricted to $Y\cap DB_\theta$, shortens distances at most as $1:\sqrt {1+(1/49)^2}\ge 0.99$.

\end{thm}

\begin{proof}
For each $z\in DB_\theta$, integrating the appropriate second derivatives of $f$ along the real segment $[0,z]$ (whose length does not exceed $\sqrt 2 \theta$) we get for $j=1, \dots, n-1$,
\be \label{eq:deriv.bd}
\left|\frac{\partial f(z)}{\partial z_j}\right| \le \sqrt {2n} M\theta=\frac{\eta}{50\sqrt {n-1}} ~, \quad \left|\frac{\partial f(z)}{\partial z_n}-\eta \right|\le \frac{\eta}{50\sqrt {n-1}}.
\ee

In particular, $|\frac{\partial f(z)}{\partial z_n}-\eta|\le \frac{\eta}{50}$, and hence $|\frac{\partial f(z)}{\partial z_n}|\ge \frac{49}{50}\eta$. Applying to $f$ the standard (local) implicit function theorem at each point $z\in Y\cap DB_\theta$, we conclude that there is a neighborhood $V\subset \bar B_\theta$ of $\bar z$ such that over $V$ the set $Y=\{f(z)=0\}$ is the graph of a regular analytic function $z_n=\phi(z_1, \dots, z_{n-1})$. It is easy to see that these local functions define in fact a unique regular analytic function $z_n=\phi(z_1, \dots, z_{n-1})$ over the entire ball $\bar B_\theta$. Indeed, on each line $L$ parallel to $Oz_n$ inside $DB_\theta$ and for any $z_n, z'_n \in L$ we have, via \eqref{eq:deriv.bd}
$$
f(z_1,\ldots,z_{n-1},z_n)-f(z_1,\ldots,z_{n-1},z'_n)\approx \eta(z_n-z'_n)
$$
and hence $f$ may have on $L$ at most one zero inside $DB_\theta$.

\smallskip

Next, by the chain rule and by \eqref{eq:deriv.bd} we have
$$
\left|\frac{\partial \phi(0)}{\partial z_j} \right|= \left|\frac{\partial f(z)}{\partial z_j}/\frac{\partial f(z)}{\partial z_n}\right| \le \frac{1}{49\sqrt {n-1}}~, \quad j=1,\ldots,n-1,
$$
and hence $||\nabla \phi(z)|| \le \frac{1}{49}$. For each $\bar z \in \bar B_\theta$ integrating along the real segment $[0,\bar z]$ we get $|z_n|=|\phi(\bar z)| \le \frac{1}{49}\theta$, and hence inside the diskoball $DB_\theta$ the hypersurface $Y$ is contained in a tabular neighborhood $W_\nu$ of the coordinate hyperplane $z_n=0$, of the size $\nu=\frac{\theta}{49}$. Finally, for $z_n'=\phi(\bar z')$ and $z_n''=\phi(\bar z'')$ integrating along the real segment $[\bar z',\bar z'']$ we get $|z_n'-z_n''| \le \frac{1}{49}||\bar z'-\bar z''||$. So the projection $\pi: Y\to \bar B_\theta$ shortens distances at most as $1:\sqrt {1+(1/49)^2}\ge 0.99$. This completes the proof of Theorem \ref{thm:impl.fn}.
\end{proof}

In order to use Theorem \ref{thm:impl.fn}, we recall that by Markov's inequality the second partial derivatives of $P$ are bounded for $z\in Q$ by $n d^4 \|P\|_1=n d^4$. So, we put $M=n d^4$. Now, we cover $Y$ with diskoballs $DB^j_{r_j}$, centered at $z_j\in Y$, whose radii $r_j$ satisfy the requirement $r_j\le{\|\nabla P (z_j)\|}/{50M\sqrt {2n(n-1)}}$ of Theorem \ref{thm:impl.fn}. To build $DB^j$ we first apply Theorem \ref{thm:beta-doubling-covering} to $Q_\delta=Q\setminus\Sigma_\delta$, with
$$
\gamma=\frac {600M\sqrt {2n(n-1)}}{K}+1=\frac {600nd^4\sqrt {2n(n-1)}}{K}+1.
$$

This theorem provides a $\gamma$-doubling ball covering $\U$ of $Q_\delta$ in $\C^n$, with $\kappa({\U}) \le m(3\sqrt {2n}\gamma)^{2n}\log(6n\gamma/\delta)$ (since by assumptions of Theorem \ref{thm:main.hypers.1} the number of critical points $w_j$ of $P$ in $\Sigma$ is $m \le (d-1)^n$, while the real dimension is $2n$). This yields
$$
\kappa({\U}) \le \frac{C_1(n,d)}{K^{2n}}\log(\frac{C_2(n, d)}{K\delta})
$$
where we can put, taking into account that $K\leq nd^4$, and after some simplifications, $C_1(n,d)=(4000n^2d^5)^{2n}$, $ C_2(n,d)=6000 n^3d^4$. This is the complexity bound required in Theorem \ref{thm:main.hypers.1}, so it remains to construct the required covering of $Y$ subordinated to the covering $\U$ of $Q_\delta$ in $\C^n$.

\begin{lem} \label{lem:grad.size}
Let $B^j\in {\U}$ be a ball of radius $R_j$, and let $z\in B^j$. Then, we have
$$
\|\nabla P(z)\|\ge 600\sqrt {2n(n-1)}MR_j.
$$
\end{lem}

\begin{proof}
Since, by definition of the $\gamma$-doubling ball covering $\U$ the concentric ball $\tilde B^j$ to $B^j$ of radius $\gamma R_j=(\frac {600M\sqrt {2n(n-1)}}{K}+1)R_j$ does not touch $\Sigma$, we conclude that for each $z\in B^j$ we have $\dist (z, \Sigma)\ge\frac {600M\sqrt {2n(n-1)}}{K}R_j$. By \eqref{eq:grad.P.dist} we obtain
$$
\|\nabla P(z)\|\ge K \dist(z, \Sigma) \ge 600M\sqrt {2n(n-1)}R_j.
$$
\end{proof}

Now, we proceed as follows, consider all the balls $B^j\in {\U}$, which intersect $Y$. For each $B^j$ we fix a point $z^j\in B^j\cap Y$. Applying a unitary coordinate transformation, we define a new coordinate system $(v_1, \dots, v_n)$ at $z^j$, such that the direction of the last coordinate axis $Ov_n$ coincides with the direction of $\nabla P(z^j)$.

\smallskip

Finally, we fix a diskoball $DB^j=DB^j_{r_j}$, with respect to $(v_1, \dots, v_n)$, centered at $z^j$, with $r_j=12R_j$. By Lemma \ref{lem:grad.size} we have
$$
\frac{\|\nabla P(z_j)\|}{50M\sqrt {2n(n-1)}}\ge\frac{600\sqrt {2n(n-1)} MR_j}{50M\sqrt {2n(n-1)}}=12R_j=r_j.
$$

So, the conditions of Theorem \ref{thm:impl.fn} are satisfied for the polynomial $P(z)-P(z_0)=P(z)-c$ on the diskoball $DB^j$. We conclude that $Y\cap DB^j_{r_j}$ is a graph of a regular analytic function $v_n=\phi_j(v_1, \dots, v_{n-1})$, such that $\|\nabla\phi_j(v_1, \dots, v_{n-1})\|\le\frac{1}{49}$ on the ball $\bar B^j_{r_j}$. Denote by $\tilde\phi_j$ the corresponding mapping of $\bar B^j_{r_j}$ to $Y$:
$$
\tilde\phi_j(v_1, \dots, v_{n-1})=(v_1, \dots, v_{n-1}, \phi_j(v_1, \dots, v_{n-1})).
$$

Finally, we apply a linear mapping $\lambda_j$ of the unit ball $\bar B_1$ to the concentric ball $\bar B^j_{r_j/4}$ of a four time smaller size, and define a chart $\psi_j$ in the covering ${\U}_Y$ of $Y$ under construction as $\psi_j=\tilde\phi_j\circ\lambda_j$. It remains to show that the images $U_j=\psi_j(B_1)$ of the charts $\psi_j$ form a doubling covering of $Y\cap Q$ with the required properties.

\smallskip

1) Clearly, $\psi_j$ are extendable from $\bar B_1$ to $\bar B_4$ and remain there univalent. Indeed, with $\lambda_j$ we shrink four times the domain, provided by the implicit function theorem.

\smallskip

2) The charts $U_j=\psi_j(B_1)$ form a covering of $Y\cap Q$. Indeed, put $\hat DB^j=DB^j_{r_j/4}=DB^j_{3R_j}$. We have $U_j=Y\cap\hat DB^j$. But the diskoball $\hat DB^j$ contains the ball $B^j_{2R_j}$, and already the balls $B^j_{R_j}$ of $\U$ cover $Q^n_\delta$. Since, by conditions $Y\subset Q_\delta$, we conclude that the diskoballs $\hat DB^j$, intersecting $Y$, and hence the charts $U_j=Y\cap\hat DB^j$, form a covering of $Y$.

\smallskip

3) For each two points $u_1, u_2$ belonging to the same connected component of $Y\cap Q$ there is a chain $Ch$ in ${\U}_Y$ joining $u_1$ and $u_2$, with $\rho(U_{j_1},U_{j_2})\ge 1/10$ for each couple of subsequent charts $U_{j_1},U_{j_2}$. Indeed, consider a curve $\omega$ joining $u_1$ and $u_2$ in $Y\cap Q$. As in the proof of Theorem \ref{thm:beta-doubling-covering} we can assume that $\omega$ does not touch faces of real dimension smaller than $2n-1$ in the sub-cubes constructed in the proof of Theorem \ref{thm:beta-doubling-covering}. (If necessary, we bring the coordinate system in $\C^n$ in general position with respect to $Y$). Marking the subsequent sub-cubes along $\omega$ and omitting repetitions, we obtain, taking the corresponding balls in $\U$, a chain of balls $B^{j_s}$ in ${\U}$, which covers $\omega$, such that for each couple of subsequent balls the intersection $Y\cap Q\cap B^{j_s}\cap B^{j^{s+1}}$ is non-empty, while the corresponding sub-cubes intersect along a common $(n-1)$-face. By the construction, for the chain $Ch$ of charts $U_{j_s}$ in ${\U}_Y$ with the same indices, each couple of subsequent charts $U_{j_s}$ has a non-empty intersection.

\smallskip

4) Consider now a couple of subsequent charts, say, $U_1, U_2$, in $Ch$. By our construction, for the corresponding balls $B^1$ and $B^2$ in $\U$ the intersection $Y\cap Q\cap B^1\cap B^2$ is non-empty. By Theorem \ref{thm:beta-doubling-covering}, the ratio of the radii $R_1$ and $R_2$ is $\frac{1}{2}$, $1$, or $2$, and their intersection contains a ball of radius at least $\frac{1}{3}$ of the smallest of $R_1$ and $R_2$. Now, we notice that the intersection of the charts $U_1, U_2$ on $Y$ contains the intersection $B^1_{2R_1}\cap B^2_{2R_2}\cap Y$ of the twice larger ball concentric to $B^1$ and $B_2$. On the other hand, by Theorem \ref{thm:impl.fn}, inside the diskoball $DB^1$ the hypersurface $Y$ is contained in a tabular neighborhood $W_\nu$ of the size $\nu=r_1/49$ of the coordinate hyperplane $v_n=0$. The inverse mapping $\psi_1^{-1}:U_1\to\hat B_1$ is just the projection to the first $n-1$ coordinates $v_1, \dots, v_{n-1}$, and by Theorem \ref{thm:impl.fn}, it shortens distances at most to a factor $0.99$. Now, an easy calculation, shows that $\rho(U_1, U_2)\ge\frac{1}{10}$.

This completes the proof of Theorem \ref{thm:main.hypers.1}.
\end{proof}

\section{Doubling inequalities on balls} \label{Sec:Doub.Ine}

\subsection{Doubling inequalities on concentric one-dimensional disks}

In this paper we work with algebraic functions. However, it is technically convenient to consider (in dimension one) a much larger class of $p$-valent functions. Let $p\in\N$, and let $f(z)$ be an analytic function in a domain $W\subset\C$. The function $f(z)$ is said to be $p$-valent in $W$ if the equation $f(z)=c$ has at most $p$ roots for any complex $c$. The study of $p$-valent functions is a classical topic in complex analysis (see \cite{Hay} and references therein).

\smallskip

The following theorem presents one of possible accurate formulations of the connection between $p$-valency and doubling inequalities, which is convenient for our purposes. For more general settings see e.g. \cite{RY, VdP}. One can get sharper constants replacing $p$-valent functions with $(s, p)$-valent ones, as defined in \cite{FY}, but we try to keep analytic tools to the minimum in this paper.

\begin{thm} \label{thm:Bier.RY.VdP}
Let $1>\alpha >\beta > 0$, and let $f(z)$ be $p$-valent in the disk $D$. Then, $f(z)$ satisfies a doubling inequality with respect to the disks $\beta D\subset\alpha D\subset D$ so that
$$
DC_f (\alpha D, \beta D)\le ((p+1)\alpha^p +{A'_p}/{(1-\alpha)^{2p+1}})/\beta^p =: c_p(\alpha, \beta)
$$
where $A'_p$ depends only on $p$.
\end{thm}

\begin{proof}
First, we recall the classical result of Biernacki \cite{Bie}:

\begin{prp} [Biernacki] \label{prp:Bier}
Let $f(z)=\sum_{k=0}^\infty a_k z^k$ be $p$-valent in the disk $D_1$. Then, for any $k\ge p+1$
$$
|a_k|\le A_p\max_{i=0, \dots, p} |a_i| k^{2p-1}
$$
where $A_p$ depends only on $p$.
\end{prp}

Via rescaling it is enough to consider only the unit disk $D=D_1$. Put $m=\max_{\beta D}|f|$. By Cauchy formula, applied to $\beta D$ for any $k$ we have $|a_k|\le{m}/{\beta^k}$. Hence, by Proposition \ref{prp:Bier} for $k\ge p+1$ we get
$$
|a_k|\le A_p k^{2p-1} m/\beta^p.
$$

Now, we obtain an upper bound for $|f|$ on $\alpha D$
\begin{align*}
\max_{\alpha D} |f(z)| &\le\sum_{k=0}^\infty |a_k|\alpha^k=\sum_{k=0}^p |a_k|\alpha^k+\sum_{k=p+1}^\infty |a_k|\alpha^k \\
&\le\sum_{k=0}^p\frac{m\alpha^k}{\beta^k}+\sum_{k=p+1}^\infty\frac{A_p k^{2p-1}m\alpha^k}{\beta^p} \\
&=\left(\frac{1-(\alpha/\beta)^{p+1}}{1-\alpha/\beta} +\frac{A_p}{\beta^p}\sum_{k=p+1}^\infty k^{2p-1}\alpha^k\right)\max_{\beta D} |f|.
\end{align*}

Now, we shall analyse the constant above. First, let us recall the polylog function $\operatorname{Li}_{s}(z)=\sum_{k=1}^\infty\frac{z^k}{k^s}$. Note that the infinite sum $\sum_{k=p+1}^\infty k^{2p-1}\alpha^k$ is in fact a tail of $\operatorname{Li}_{s}(z)$ with the parameters $s=1-2p$ and $z=\alpha$. In this case, when $s=-n$ for $n\in\N$, we have the following formula (e.g. see \cite{Mil})
$$
\operatorname{Li}_{-n}(z)=\frac{1}{(1-z)^{n+1}}\sum_{k=1}^n a_{n-1, k-1} z^k
$$
where the coefficients (aka Eulerian numbers) can be obtained by the recurrence equation $a_{n, k}=(n+1-k) a_{n-1, k-1}+ka_{n-1, k}$. Thus, for $s=1-2p$ and $z=\alpha$, we get
\begin{align*}
\frac{1-(\alpha/\beta)^{p+1}}{1-\alpha/\beta} & +\frac{A_p}{\beta^p}\sum_{k=p+1}^\infty k^{2p-1}\alpha^k\le \\
&\le\frac{1-(\alpha/\beta)^{p+1}}{1-\alpha/\beta} +\frac{A_p}{\beta^p(1-\alpha)^{2p}}\sum_{k=1}^{2p-1} a_{2p, k-1}\alpha^k \\
&\le\frac{1-(\alpha/\beta)^{p+1}}{1-\alpha/\beta} +\frac{A'_p}{\beta^p(1-\alpha)^{2p}}\sum_{k=1}^{2p-1}\alpha^k \\
&\le\frac{1-(\alpha/\beta)^{p+1}}{1-\alpha/\beta} +\frac{A'_p}{\beta^p(1-\alpha)^{2p}}\cdot\frac{1-\alpha^{2p}}{1-\alpha} \\
&\le(p+1)(\alpha/\beta)^p +\frac{A'_p}{\beta^p(1-\alpha)^{2p+1}}=c_p(\alpha, \beta)
\end{align*}
where $A'_p$ is another constant, which depends only on $p$.
\end{proof}

\subsection{Concentric higher-dimensional balls}

We consider analytic functions $f(z_1, \dots, z_n)$ of complex variables $z=(z_1, \dots, z_n)$. Let $W\subset\C^n$ be a domain. A function $f(z)$ analytic on $W$ is called sectionally $p$-valent, if it possesses the following property: for each straight line $L$ the restriction $f_L$ of $f$ to $L\cap W$ is $p$-valent. Algebraic functions $f$ of degree $d$ are sectionally $p(d)$-valent by the B\'ezout theorem, as well as other important classes of functions.

\begin{thm} \label{thm:Bier.RY.1}
Let $1>\alpha >\beta > 0$, and let $f(z)$ be sectionally $p$-valent in the ball $B\subset\C^n$. Then, $f(z)$ satisfies a doubling inequality with respect to the balls $\beta B\subset\alpha B\subset B$, with the doubling constant $DC_f(\alpha B, \beta B)\le c_p(\alpha, \beta)$.
\end{thm}

\begin{proof}
Let $z\in\alpha B$. Consider the complex straight line $L$ passing through the points $0$ and $z$, and let $f_L$ be the restriction of $f$ to $L$.
%
%
Now, applying Theorem \ref{thm:Bier.RY.VdP} to $f_L$ with $\beta B\cap L\subset\alpha B\cap L\subset B\cap L$, we obtain the required inequality for $f_L(z)=f(z)$
$$
|f(z)| \le\max_{\alpha B\cap L}|f_L|\le DC_f(\alpha B\cap L, \beta B\cap L)\max_{\beta B\cap L}|f_L|\le c_p(\alpha, \beta)\max_{\beta B} |f|.
$$
\end{proof}

\subsection{Doubling inequalities on non-concentric balls}

Now, we extend the doubling inequality for sectionally $p$-valent functions, provided by Theorem \ref{thm:Bier.RY.1} to couples of non-concentric balls.

\begin{cor} \label{cor:non-conc}
Let $f(z)$ be sectionally $p$-valent in the ball $B_{4}$, and let $B_1$ be the concentric ball. Let $\Delta_\rho\subset B_1$ be a ball of radius $\rho$ in $B_1$, not necessarily concentric to it. Then,
\begin{align*}
DC_f(B_1, \Delta_\rho)\le c_p/\rho^p
\end{align*}
where $c_p>0$ depends only on $p$.
\end{cor}

\begin{proof}
Let $\Delta_{\hat\rho}\subset B_4$ be the maximal sub-ball of $B_4$ concentric to $\Delta_\rho$. We have $3<\hat\rho<4$. Consider now the ball
$\Delta_{\hat\rho/2}$ concentric to $\Delta_\rho$. By Theorem \ref{thm:Bier.RY.1}, applied to the concentric balls
$\Delta_\rho\subset\Delta_{\hat\rho/2}\subset\Delta_{\hat\rho}$ we get
$$
\max_{\Delta_{\hat\rho/2}} |f|\le c_p(\hat\rho/2\hat\rho, \rho/\hat\rho)\max_{\Delta_{\rho}} |f|.
$$

Now, notice that $\Delta_{\hat\rho/2}$ contains the disk $B_{1/2}$ concentric to $B_1$. 
Hence, $\max_{B_{1/2}} |f|\le\max_{\Delta_{\hat\rho/2}} |f|$. Once more, By Theorem \ref{thm:Bier.RY.1}, applied to the concentric disks $B_{1/2}\subset B_1\subset B_4$, we conclude that
$$
\max_{B_1} |f|\le c_p (1/4, 1/8)\max_{B_{1/2}} |f|.
$$

Taking these two inequalities into account yields
$$
DC_f(B_1, \Delta_\rho)\le c_p(1/4, 1/8) c_p(1/2, \rho/\hat\rho) .
$$

Finally, simply observe that the constant $c_p(1/4, 1/8) c_p(1/2, \rho/\hat\rho)$ can be written as $c_p/\rho^p$, where $c_p>0$ depends only on $p$.
\end{proof}

\section{Doubling inequalities on complex manifolds} \label{Sec:Gen.Doub.Ine}

In this section we give a very general form of a doubling inequality for analytic functions $f$ on a manifold $Y$ that are sectionally $p$-valent with respect to a certain fixed doubling covering $\U$ of $Y$. We do so via continuation along chains of charts in $\U$.

An analytic function $f$ on $Y$ is called sectionally $p$-valent with respect to the doubling covering $\U$ of $Y$ if for each chart $U_j$ in $\U$ the function $f_j=f\circ\psi_j$ is sectionally $p$-valent in $B_4$. Certainly, polynomials or algebraic functions on algebraic manifolds $Y$ satisfy this property for each covering $\U$ with algebraic charts, with $p$ depending only on the degrees of the algebraic objects involved.

\begin{thm} \label{thm:uniform-doubling-covering}
Let $Y$ be a complex manifold, $\Omega\subset G$ be compact domains in $Y$, and $z\in G$. Let $f$ be an analytic function in a neighborhood of $G$ in $Y$, and let $\U$ be a doubling covering of $G$ in $Y$ such that $f$ is sectionally $p$-valent with respect to $\U$. Then, we have
$$
|f(z)|\le K(z, \Omega, f)\max_{\Omega}|f|
$$
where
$$
K(z, \Omega, f)=\inf_{Ch\in CH(z, \Omega, {\U})}\frac{c_p^{\ell(Ch)}}{\rho(U_{j_1}, \Omega)^p\prod_{m=1}^{\ell(Ch)-1}\rho(U_{j_m}, U_{j_{m+1}})^p}
$$
and $c_p>0$ being the constant from Corollary \ref{cor:non-conc}.
\end{thm}

\begin{proof}
By the assumptions, for each chart $U_j$ of $\U$ the function $f_j=f\circ\psi_j$ is sectionally $p$-valent in $B_4$. Let $Ch=\{j_1, \dots, j_n\}$ be a chain in $CH(z, \Omega, \U)$. By renaming the indices we may assume that $Ch=\{1, \dots, n\}$. By the definition of $\rho(U_{j}, U_{j+1})$, there is a subball $\Delta_{\rho(U_{j}, U_{j+1})}$ of radius $\rho(U_{j}, U_{j+1})$, such that
$$
\Delta_{\rho(U_{j}, U_{j+1})}\subset\psi^{-1}_{j+1}(U_{j}\cap U_{j+1})\subset B_1.
$$

Thus, by the definition of $f_{j+1}$ we have
$$
\max_{\Delta_{\rho(U_{j}, U_{j+1})}}|f_{j+1}|\le\max_{\psi^{-1}_{j+1}(U_{j}\cap U_{j+1})}|f_{j+1}|=\max_{U_{j}\cap U_{j+1}} |f|\le\max_{U_{j}}|f|.
$$

Now, applying Corollary \ref{cor:non-conc} to $f_{j+1}$, we have
$$
\max_{B_1}|f_{j+1}|\le\frac{c_p}{\rho(U_j, U_{j+1})^p}\max_{\Delta_{\rho(U_{j}, U_{j+1})}}|f_{j+1}|.
$$

Thus, by combining these two inequalities, we conclude
$$
\max_{U_{j+1}} |f|=\max_{B_1}|f_{j+1}|\le\frac{c_p}{\rho(U_j, U_{j+1})^p}\max_{U_{j}}|f|.
$$

This allows us to pass from one chart to the next along the chain. Verbally repeating this calculation, as we pass from $\Omega$ to $U_1$ in the chain, we get for each chain $Ch\in CH(z, \Omega, {\U})$
$$
|f(z)|\le\max_{U_{n}}|f|\le\frac{c_p^{\ell(Ch)}}{\rho(U_{1}, \Omega)^p\prod_{m=1}^{\ell(Ch)-1}\rho(U_{m}, U_{{m+1}})^p}\max_{\Omega}|f|.
$$

Taking infimum over all the chains in $CH(z, \Omega, {\U})$ completes the proof of Theorem \ref{thm:uniform-doubling-covering}.
\end{proof}

Let us give a weaker, but more simple version of Theorem \ref{thm:uniform-doubling-covering}. We assume that $\Omega\subset G\subset Y$, and $f$ as before, and fix a certain doubling covering $\U$ of $G$ in $Y$, such that $f$ is sectionally $p$-valent with respect to $\U$.

Let us make the following assumption on $\U$: there are constants $\ell(\U) \le \kappa(\U)$, and $\rho(\U)>0$, such that any two points in the same connected component of $G$ can be joined by a chain $Ch$ in $\U$ of the length $\ell(Ch) \le \ell(\U)$, with any two subsequent charts $U_i, U_j$ in $Ch$ satisfying $\rho(U_i, U_j) \ge \rho(\U)$. This condition is satisfied in our main results below.
%
%
%
Assuming in addition that $\rho(\U, \Omega), \rho(\U)\ge\rho$, we have the following simple and natural corollary of Theorem \ref{thm:uniform-doubling-covering}.

\begin{cor} \label{cor:simple}
Let $f$ be an analytic function in $Y$. Let $\U$ be a doubling covering, such that $f$ is sectionally $p$-valent with respect to $\U$. Assume that $\rho(\U, \Omega), \rho(\U)\ge\rho$. Then, we have
$$
DC_f(G, \Omega) \le\left({c_p}/{\rho^p}\right)^{\ell(\U)} \le\left({c_p}/{\rho^p}\right)^{\kappa(\U)} .
$$
\end{cor}

Finally, we use Corollary \ref{cor:simple} to reverse the inequality, obtaining a lower bound on the number of charts in doubling coverings in terms of the doubling constant for certain functions.

\begin{cor} \label{cor:reverse}
Let $f$ be an analytic function in $Y$. Let $\U$ be a doubling covering, such that $f$ is sectionally $p$-valent with respect to $\U$. Assume that $\rho(\U, \Omega), \rho(\U)\ge\rho$. Then, we have
$$
\kappa(\U)\ge\frac{\log DC_f(G, \Omega)}{\log(c_p/\rho^p)}.
$$
\end{cor}

\subsection{Doubling inequality for polynomials on $Y$} \label{Sec:doub.on.Y}

As an immediate consequence of Theorem \ref{thm:main.hypers.1} we obtain an explicit bound in a doubling inequality for polynomials $S$ of degree $d_1$ on hypersurfaces $Y$. Let $P(z)$, $Y$, $\Sigma=\Sigma(P)$, $G=Y\cap Q$, $\delta=\dist(G, \Sigma(P))>0$ be as above, and let ${\U}_Y$ be the doubling covering of $G$ in $Y$ constructed in Theorem \ref{thm:main.hypers.1}. Let $\Omega\subset G$ be a compact sub-domain of $G$. To simplify the presentation we shall assume that $\rho(\U, \Omega)\ge\frac{1}{10}$.

\begin{cor} \label{cor:simple1}
Let $Y, G, \Omega$ be as above. Let $f$ be a restriction of a polynomial $S$ of degree $d_1$ to $Y$. Then, we have
$$
DC_f(G, \Omega)\le({C_2(n, d)}/{K\delta})^{C_3(n, d, d_1)/K^{2n}}.
$$
\end{cor}

\begin{proof}
By Theorem \ref{thm:main.hypers.1} we have $\kappa ({\U})\le (C_1(n, d)/K^{2n})\log({C_2(n, d)}/{K\delta})$. We also have, by Theorem \ref{thm:main.hypers.1}, $\rho=\min(\rho({\U}, \Omega), \rho({\U}))\ge\frac{1}{10}$. Thus, in order to apply Corollary \ref{cor:simple}, we need to study the
valency $p$ of the restrictions of $S\circ\psi_j$ to the straight lines $L$ in $\bar B_1$, for a polynomial $S$ of degree $d_1$ on $Y$.

\smallskip

So, we have to bound the number of solutions of $S\circ\psi_j=h$ on such lines. Let $L$ be defined in the subspace $\C^{n-1}\subset\C^n$ by the affine equations $l_i=0$, $i=1, \dots, n-2$, which we extend to $\C^n$ via the projection $\pi$. The solutions of $S\circ\psi_j=h$ on $L$ are in a one-to-one correspondence with the points, defined in $\C^n$ by the system of $n$ equations
$$
l_i=0~, i=1, \dots, n-2~, P=c~, S=h
$$
of degrees $1, d, d_1$, respectively. By B\'ezout theorem, this number is at most $dd_1$. So, we set $\rho=\frac{1}{10}$, $p=dd_1$ in Corollary \ref{cor:simple}, and obtain
$$
DC_f(G, \Omega)\le(10^p c_p)^{(C_1/K^{2n})\log({C_2}/{K\delta})}=({C_2}/{K\delta})^{C_3/K^{2n}}
$$
where $C_3=C_3(n, d,d_1)=\log(10^p c_p) C_1(n, d)$, with $p=dd_1$.
\end{proof}

\section{Concluding remarks} \label{Sec:examples}

\subsection{The hyperbola $H_\e=\{zy=\e^2\}$ in $\C^2$, and similar curves}

The example of hyperbola $H_\e=\{xy=\e^2\}$ in $\C^2$ plays a prominent role in study of smooth parameterizations. Already $C^k$-parameterization of the real hyperbola $Q^{real}_\e=H_\e\cap I^2$, for $k\ge 2$ is a nontrivial question. Finding an exact number of charts in $C^2$-parameterization of $H^{real}_\e$ was suggested as an exercise in \cite{Gro}. This exercise was partially completed in \cite{Yom4}, where also some initial results on the complexity of analytic parameterizations of $H^{real}_\e$ were obtained. How many mild charts (with fixed parameters) do we need to cover $H^{real}_\e$, as $\e$ tends to zero, is an open question. Application of Theorem \ref{thm:main.hypers.1}, and of Corollary \ref{cor:reverse} provide the following result:

\begin{thm}
There is a doubling covering $\U$ of $G_\e=H_\e\cap Q$ in $H_\e$ such that $\rho(\U)\ge\frac{1}{10}$ and $\kappa(\U)\le c_1\log({c_2}/{\e})$, where $c_1, c_2$ are absolute constants. For any doubling covering $\tilde {\U}$ of $G_\e$ with $\rho(\tilde {\U})\ge\frac{1}{10}$, we have $\kappa(\tilde {\U})\ge c_3\log({1}/{\e})$.
\end{thm}

\begin{proof}
The polynomial $P(z, y)=zy-\e^2$, defining $H_\e$, has the only singular point at the origin $0\in\C^2$, and the norm of $\nabla P(z, y)=(y, z)$ is exactly the distance $|z|^2+|y|^2$ from the point $(z, y)$ to the origin. So, $K=K(P)=1$. Now, the distance of $H_\e$ to the origin is equal to $\sqrt 2\e$, and it is achieved along the ``vanishing cycle'' $z=\e e^{i\theta}, y=\e e^{-i\theta}$. Application of Theorem \ref{thm:main.hypers.1} provides the required doubling covering $\U$ of $G_\e$, with $\kappa(\U)\le c_1\log({c_2}/{\e})$, where $c_1, c_2$ are absolute constants.

\smallskip

Consider now a linear polynomial $y$ restricted to $H_e$. Its maximal absolute value on $G$ is one. Put $\Omega=\{|z|\ge\frac{1}{2}\}\cap G$. We have $\max_\Omega |y|=2\e^2$. Therefore, we get $DC_y(G, \Omega)=\frac{1}{2\e^2}$. By Corollary \ref{cor:reverse}, we conclude that in any doubling covering $\tilde {\U}$ of $G$ in $H_\e$, with $\rho(\tilde {\U}), \rho({\U}, \Omega)\ge\frac{1}{10}$ the number of charts $\kappa({\U})$ is at least $c_3\log({1}/{\e})$, with $c_3$ an absolute constant.
\end{proof}

In the same way we can work with more general polynomials $P(z, y)$ representable as products of regular factors. In particular, consider a polynomial
$$
P(z, y)=z(z-1)\cdots(z-d)y(y-1)\cdots(y-d)
$$
of degree $2d+2$. This polynomial has exactly $(d-1)^2$ isolated non-degenerate singular points. Proceeding as above, we construct a doubling covering of the curve $Y_\e=\{P(z, y)=\e^2\}$ in a cube $Q_{d+1}$ of size $d+1$, with an order of $(d-1)^2\log({1}/{\e})$ charts, and show that for any doubling covering the number of charts must be of the same order.

\subsection{Higher-dimensional quadrics}

Let
$$
P(z)=P(z_1, \dots, z_n)=\sum_{j=1}^n z_j^2.
$$

As for the hyperbola, $P$ has the only singular point at the origin $0\in\C^n$, and the norm of $\nabla P(z)=(2z_1, \dots, 2z_n)$ is exactly twice the distance from the point $z$ to the origin. So, $K=K(P)=2$. Consider
$$
Y_\e=\{P(z)=\e^2\}.
$$

The distance of $Y_\e$ to the origin is equal to $\e$, and it is achieved at the real points of the form $(0, 0, \dots, \pm\e, 0, \dots0)$. Indeed, for any point $(z_1, \dots, z_n)\in Y_\e$ we have
$$
\|z\|^2 = \sum_{j=1}^n |z_j|^2\ge |\sum_{j=1}^n z_j^2|=\e^2.
$$

As above, Theorem \ref{thm:main.hypers.1} produces a doubling covering $\U$ of $G=H_\e\cap Q$ with not more than $c_5\log({c_6}/{\delta})$ charts, where $c_5, c_6$ depend only on $n$.

\smallskip

Algebraic geometry of complex algebraic hypersurfaces, considered from the point of view of doubling coverings and doubling inequalities provides a variety of important phenomena. We plan to present some further results in this direction separately. In particular, it would be very interesting to estimate the covering complexity of the Brieskorn-Milnor fibers $P(z)=\sum_{j=1}^n z_j^{k_j}=\e$. However, in this case the singular point of $P$ at the origin, although isolated, is not non-degenerate any more, and Theorem \ref{thm:main.hypers.1} does not work.

\end{document}